\documentclass[preprint,12pt]{elsarticle}

\usepackage{amssymb}
\usepackage{amsmath}
\usepackage{amsfonts}
\usepackage{graphicx}
\usepackage{graphics}
\usepackage{tikz}
\usepackage{titlesec}

\newtheorem{theorem}{Theorem}

\newtheorem{corollary}[theorem]{Corollary}

\newtheorem{definition}[theorem]{Definition}
\newtheorem{example}[theorem]{Example}

\newtheorem{lemma}[theorem]{Lemma}

\newtheorem{problem}[theorem]{Problem}
\newtheorem{proposition}[theorem]{Proposition}
\newtheorem{remark}[theorem]{Remark}

\def\qed{\vbox{\hrule
 \hbox{\vrule\hbox to 5pt{\vbox to 8pt{\vfil}\hfil}\vrule}\hrule}}

\usepackage{tikz}
\usepackage{tkz-graph}

\journal{xxxxxxx}

\begin{document}

\begin{frontmatter}


\title{Block matrices and Guo's Index for block circulant matrices with circulant blocks}

\author[]{Enide Andrade\corref{cor1}}

\address{CIDMA-Center for Research and Development in Mathematics and Applications
         Departamento de Matem\'atica, Universidade de Aveiro, 3810-193, Aveiro, Portugal.}
\ead{enide@ua.pt}
\cortext[cor1]{Corresponding author}

\author{Cristina Manzaneda}

\address{Departamento de Matem\'{a}ticas, Facultad de Ciencias. Universidad Cat\'{o}lica del Norte. Av. Angamos 0610 Antofagasta, Chile.}
\ead{cmanzaneda@ucn.cl}

\author{Hans Nina}

\address{Departamento de Matem\'{a}ticas, Facultad de Ciencias B\'asicas. Universidad de Antofagasta.\\
Av. Angamos 601 Antofagasta, Chile.}
\ead{Hans.nina@uantof.cl}

\author{Mar\'{i}a Robbiano}

\address{Departamento de Matem\'{a}ticas, Facultad de Ciencias. Universidad Cat\'{o}lica del Norte. Av. Angamos 0610 Antofagasta, Chile.}
\ead{mrobbiano@ucn.cl}

\begin{abstract}

In this paper we deal with  circulant and partitioned into $n$-by-$n$ circulant blocks matrices and introduce spectral results concerning this class of matrices. The problem of finding lists of complex numbers corresponding to a set of eigenvalues of a nonnegative block matrix with circulant blocks is treated. Along the paper we call realizable list if its elements are the eigenvalues of a nonnegative matrix.
The Guo's index $\lambda_0$ of a realizable list is the minimum spectral radius such that the list (up to the initial spectral radius) together with $\lambda_0$
is realizable.
The Guo's index of block circulant matrices with circulant blocks is obtained,
and in consequence, necessary and sufficient conditions concerning the NIEP, Nonnegative Inverse Eigenvalue Problem, for the realizability of some spectra are given.
\end{abstract}


\end{frontmatter}
\section{A brief review and some tools}
In this section we present a brief resume related 
to nonnegative inverse eigenvalue problem (NIEP).  Recall that a square matrix $A=(a_{ij})$ is nonnegative ($A \geq 0$) if and only if $a_{ij} \geq 0,$ for each $ i,j= 1, \ldots,n.$ For more background material on nonnegative matrices see for example \cite{Berman}.
The NIEP is the problem of determining necessary and sufficient conditions for a list of complex
numbers
to be the spectrum of an $n$-by-$n$ 
nonnegative matrix $A$. If a list $\sigma$ is the
spectrum of a nonnegative matrix $A$, then $\sigma$ is \textit{realizable} and the matrix $A$ \textit{realizes} $\sigma$ (or, that is a realizing matrix for the list). This is a hard problem and it is considered by many authors since more than 50 years ago. This problem was firstly considered by Sule\u{\i}manova \cite{SLMNva} in $1949$.
Many partial results were found but the problem is still unsolved for $n \geq 5$. For $n=3$ it was solved in \cite{LwyLdn} and for matrices of order $n=4$ the problem was solved  in \cite{Meehan} and \cite{Mayo}. In its general form it has been studied in e.g. \cite{Boyle,GUO,Johnson,Laffey1,Laffey2,LwyLdn,Smgoc1,Smigoc2}. There are some variants of this problem namely for instance, the one called symmetric nonnegative inverse eigenvalue problem, SNIEP,  (when the nonnegative realizing matrix is required to be symmetric). This is also an open problem and some work on this can be seen in \cite{Fiedler,Laffey,LMc,Soules}.
 Another variants of the original problem is the question for which lists of $n$ real numbers can occur as eigenvalues of an $n$-by-$n$  nonnegative matrix and it is called real nonnegative inverse eigenvalue problem (RNIEP). Some results can be seen in e.g. \cite{Boro,Friedland,HzlP,RS1, SRM}. The structured NIEP is an analogous problem to NIEP where the realizing matrix must be structured, for instance, the matrix can be symmetric, Toeplitz, Hankel, circulant, normal, etc. see in \cite{Fiedler, MAR,PP} and the reference therein.

To not be so extensive on the description of this problem the reader must refer to some surveys on NIEP, for instance in \cite{Johnson2} and in the references therein.

Throughout the text, $\mathbb{C}^{k}$ and $\sigma\left(A\right)$  denote the set of complex $k$-tuples and of the eigenvalues of a square matrix $A$, respectively. Also $\rho (A)$ denotes the spectral radius of $A$. Here, the identity matrix of order $n$ is denoted by $I_n$ and if the order of the identity matrix can be easily deduced then it is just denoted by $I$.

Since a nonnegative matrix is real, its characteristic polynomial must have real coefficients and then
$ \{\lambda_{0}, \ldots, \lambda_{n-1}\} =\sigma = \overline{\sigma} =  \{\overline{\lambda_{0}}, \ldots, \overline{\lambda_{n-1}}\},$ where $\overline{\lambda}$ stands for the complex conjugate of $\lambda \in \mathbb{C}.$

Therefore consider the following definition:

\begin{definition}
The complex $n$-tuple $(\lambda_{0},\ldots,\lambda_{n-1})$ is  \textit{closed under conjugation} if the condition $$\{\lambda_{0},\ldots,\lambda_{n-1}\}=\{\overline{\lambda_{0}},\ldots,\overline{\lambda_{n-1}}\},$$ holds.
\end{definition}

The Perron-Frobenius theory of nonnegative matrices \cite{Berman,Johnson3} plays in this problem an important role.
The theory provides several and nice important necessary conditions for the NIEP. See below some of these conditions resumed. Here $s_{k} (\sigma)=\sum\limits_{i=0}^{n-1}\lambda
_{i}^{k}$ is called the $k$-th moment.

Some necessary conditions for the list  $\sigma = (\lambda_{0}, \ldots, \lambda_{n-1})$ of complex numbers to be the spectrum of a nonnegative matrix are:
\begin{enumerate}
\item The spectral radius,  $\max \left\{ \left\vert \lambda \right\vert
:\lambda \in \sigma (A) \right\} $, called the Perron eigenvalue,  belongs to $\sigma.$

\item The list $\sigma $ is closed under complex conjugation.

\item $s_{k}\geq 0$ for $ k=1,2,\ldots $

\item $s_{k}^{m}\left( \sigma \right) \leq n^{m-1}s_{km}\left( \sigma
\right)$ for $ k,m=1,2,\ldots $
\end{enumerate}
 The last condition was
proved by Johnson \cite{Johnson} and independently by Loewy and London \cite{LwyLdn}.

\vspace{0.5 cm}
The following fundamental theorem  was proven in \cite{GUO} and in its statement it is introduced formally the notion of Guo's index.
\begin{theorem}\cite [Theorem 2.1]{GUO}
\label{Guo} Let $(\lambda_{1},\ldots,\lambda_{n-1})$ be a closed under complex conjugation
$(n-1)$-tuple then, there exists a real number $\lambda_0$ (called Guo\ {'}s index) where
$$\max_{{1}\leq j\leq n-1}{|\lambda_j|}\leq\lambda_0.
$$
such that the list $(\lambda,\lambda_{1},\ldots,\lambda_{n-1})$ is realizable by an $n$-by-$n$ nonnegative matrix $A$ if and only if $\lambda\geq \lambda_0$. Furthermore, $\lambda_0\leq 2n\max_{1\leq j\leq n-1}{|\lambda_j|}$.

\end{theorem}
The next lemma due to T. Laffey and H. \v{S}migoc, \cite{Laffey-Smigoc}, allows to show that the companion matrix of a certain polynomial $f(x)$ is nonnegative. In fact, in order to do that the lemma proves that all but the leading coefficients of the polynomial $f(x)$ are less than or equal to zero.

\begin{lemma}\label{lema L-S}\cite{Laffey-Smigoc}
Let $t$ be a nonnegative real number and let $\lambda_0,\lambda_1,\ldots,\lambda_{n-1}$ be complex numbers with real parts less than or equal to zero, such that the list $(\lambda_1,\lambda_2,\ldots,\lambda_{n-1})$  is closed under complex conjugation. Let
$$\lambda_0=2t-\lambda_1-\cdots-\lambda_{n-1}
$$
and
\begin{equation}
f(x)=(x-\lambda_0)\prod\limits_{j=1}^{n-1}(x-\lambda_j)=x^{n-1}-2tx^{n-2}+\alpha_2x^{n-3}+\cdots+\alpha_{n-1}.
\end{equation}
Then, the condition $\alpha_2\leq0$ implies that $\alpha_j\leq0,$ for $j=3,\ldots,n-1$.
\end{lemma}
The following definition generalizes the concept of circulant matrix.
\begin{definition}
A square matrix of order $n$ with $n\geq 2$ is called a \textit{permutative matrix} or permutative when all its rows (up to the first one) are permutations of precisely its first row.
\end{definition}
This concept was introduced in \cite{PP}. The spectra of a class of permutative matrices was studied in \cite{MAR}.
In particular, spectral results for matrices partitioned into $2$-by-$2$ symmetric blocks were presented and, using these results sufficient conditions on a given list to be the list of eigenvalues of a nonnegative permutative matrix were obtained and the corresponding permutative matrices were constructed. Here, in \cite{MAR}, it was introduced the concept of permutatively equivalent matrix.

\begin{definition} \cite{MAR}
{\rm
Let $\mathbf{\tau }=\left( \tau _{1},\ldots ,\tau _{n}\right) $ be an $n$-tuple whose elements are permutations in the symmetric group\ $S_{n}$, with $\tau _{1}=id$.\ Let $\mathbf{a=}\left( a_{1},\ldots ,a_{n}\right) \in
\mathbb{C}^{n}$. Define the row-vector,
\begin{equation*}
\tau _{j}\left( \mathbf{a}\right) =\left( a_{\tau _{j}\left( 1\right)
},\ldots ,a_{\tau _{j}\left( n\right) }\right)
\end{equation*}%
and consider the matrix
\begin{equation}
\tau \left( \mathbf{a}\right) =%
\begin{pmatrix}
\tau _{1}\left( \mathbf{a}\right)   \\
\tau _{2}\left( \mathbf{a}\right)  \\
\vdots \\
\tau _{n-1}\left( \mathbf{a}\right)  \\
\tau _{n}\left( \mathbf{a}\right)
\end{pmatrix}
.  \label{permut}
\end{equation}
An $n$-by-$n$ matrix $A$ is called $\mathbf{\tau}$\emph{-permutative} if
$A=\tau \left( \mathbf{a}\right) $ for some $n$-tuple $\mathbf{a}$.}
\end{definition}

\begin{definition} \cite{MAR}
\label{rm}
If $A$ and $B$ are $\mathbf{\tau }$-permutative by a common vector
$\mathbf{\tau }=\left( \tau _{1},\ldots ,\tau _{n}\right)$
then, they are called \textit{permutatively equivalent}.
\end{definition}


The paper is organized as follows: In Section 2 we introduce some concepts and results related with circulant matrices and block circulant matrices. In particular, we give a new necessary and sufficient condition for the list $\sigma=(\lambda_1,-a+bi,-a-bi,\ldots,-a+bi,-a-bi)$  with $a>0, b>0$  to be the spectrum of a nonnegative circulant matrix. This result improves the one proved in \cite[Proposition 4]{RS}. We also refer the importance of circulant matrices and block circulant matrices in some applied areas. In Section 3 we  present  spectral  results for  matrices  partitioned into blocks  where  each  block  is  a  square  circulant  matrix  of  order $n$ then, the spectral results are applied to structured NIEP and SNIEP.
In Section 4  some properties of a matrix partitioned into blocks with a certain structure are found using some already known structure on matrices of smaller size.
Finally at Section 5 it is established the  Guo's index for block circulant matrices with circulant blocks. Throughout the paper some illustrative examples are presented.

\section{Circulant matrices and block circulant matrices}

The class of circulant matrices and their properties are introduced in \cite{Davis}.
In \cite{Karner} it was presented a spectral decomposition of four types of real
circulant matrices. Among others, right circulants (whose elements topple
from right to left) as well as skew right circulants (whose elements change
their sign when toppling) were analyzed. The inherent periodicity of
circulant matrices means that they are closely related to Fourier analysis
and group theory.

Let $a=\left( a_{0},a_{1},\ldots ,a_{m-1}\right) ^{T}\in \mathbb{R}^{m}$ be given.

\begin{definition} \cite{Davis, Karner}
A \emph{\ real right circulant matrix} (or simply, \emph{circulant matrix}), is a matrix of the form
\begin{equation*}
A\left( a\right) =
\begin{pmatrix}
a_{0} & a_{1} & \ldots  & \ldots & a_{m-1} \\
a_{m-1} & a_{0} & a_{1} & \ldots & a_{m-2} \\
a_{m-2} & \ddots  & \ddots  & \ddots  & \vdots  \\
\vdots  & \ddots  & \ddots  & a_{0} & a_{1} \\
a_{1} & \ldots  & a_{m-2} & a_{m-1} & a_{0}
\end{pmatrix}
\end{equation*}
where each row is a cyclic shift of the row above to the right.
\end{definition}

The matrix $A\left( a\right)$ is clearly determined by its first row. Therefore, the above circulant matrix is also sometimes denoted by $circ(a_0,a_1,\ldots,a_{m-1})$ or, in a more simple way by $(a_0,a_1,\ldots,a_{m-1}).$
The next concepts can be seen in \cite{Karner}.
The entries of the unitary discrete Fourier transform (DFT) matrix $F=\left(
f_{pq}\right) $ are given by
\begin{equation}
\label{fourier-m}
f_{pq}:=\frac{1}{\sqrt{n}}\omega ^{pq},\ p=0,1,\ldots ,n-1,\ q=0,1,\ldots
,m-1,
\end{equation}%
where
\begin{equation}
\label{omega_root}
\omega =\cos \frac{2\pi }{m}+i\sin \frac{2\pi }{m}.
\end{equation}%

The following results characterize the circulant spectra.

\begin{theorem}
\label{teorema 22}
\cite{Karner} Let $a=(a_0,\ldots,a_{m-1})$ and $A(a)=circ(a_0,\ldots,a_{m-1}).$ Then $$A\left( a\right) =F\Lambda \left( a\right) F^{\ast },$$ with
$$
\Lambda \left( a\right) =diag\left( \lambda _{0}\left( a\right) ,\lambda
_{1}\left( a\right) ,\ldots ,\lambda _{m-1}\left( a\right) \right) $$
and
\begin{eqnarray}
\label{eigenvalues}
\text{\ }\lambda _{k}\left( a\right) =\sum\limits_{\ell=0}^{m-1}a_{\ell}\omega
^{k\ell}\text{,\quad\ }k=0,1,\ldots ,m-1.
\end{eqnarray}
\end{theorem}

\begin{corollary}
\label{fund} Let $a$ defined as in Theorem \ref{teorema 22} and consider $$v:=v(a)=\left( \lambda _{0}\left( a\right) ,\lambda _{1}\left(
a\right) ,\ldots ,\lambda _{m-1}\left( a\right) \right) ^{T}.$$
Then,
\begin{eqnarray}
\label{coefficients}
a_{k}=\frac{1}{m}\sum\limits_{\ell=0}^{m-1}\lambda _{\ell}\omega ^{-k\ell}\text{%
,\quad } k=0,1,\ldots ,m-1.
\end{eqnarray}
\end{corollary}

\begin{remark}

Let $a=\left(  a_{0},\ldots,a_{m-1}\right)  ^{T}$. By Corollary \ref{fund}
 \[a=\frac{1}{\sqrt{m}}F^{\ast}v(a)\quad \text{with}\quad v(a)=\sqrt{m}Fa.\]

\end{remark}

The next proposition obtains the Guo's index of some spectra and it is a generalization of the result obtained by O. Rojo and R. Soto in  \cite[Proposition 4]{RS}.

\begin{proposition}
Consider the complex list of $n$ elements $$\sigma=(\lambda_1,-a+bi,-a-bi,\ldots,-a+bi,-a-bi)$$ with $a>0, b>0$ and $\lambda_1\geq\sqrt{a^2+b^2}$. Then, $\sigma$ is realizable by an $n$-by-$n$ matrix $A$ if and only if
\begin{equation}
    \lambda_1\geq(n-1)a+n\max\bigg\{0,\frac{b}{\sqrt{n}}-a\bigg\}.
\end{equation}
\end{proposition}

\begin{proof}
Let $s\geq 0$ and consider $\sigma'=\{(n-1)(a+s),-(a+s)+bi,-(a+s)-bi,\ldots,-(a+s)+bi, -(a+s)-bi\}$. We claim that there exists an $s\geq 0$ such that a companion matrix
$$
B=\begin{pmatrix}
0 & 1 & 0 & \ldots & 0 & 0\\
0 & 0 & 1 & \ldots & 0 & 0\\
\vdots & \vdots & \ddots & \ddots & \vdots & \vdots\\
0 & 0 & 0 & \ldots & 0 & 1\\
b_n & b_{n-1} & b_{n-2}& \ldots & b_2 & 0\\
\end{pmatrix}
$$
realizes $\sigma'$. In fact, for the characteristic polynomial of $B$ we have $$p(x)=x^n-b_2(\sigma)x^{n-2}-\cdots-b_n(\sigma),$$
and from the Newton identities, \cite{Johnson2},
$$
b_2(\sigma)=-\frac{1}{2}[(n-1)^2(a+s)^2+(n-1)(a+s)^2-(n-1)b^2].
$$
As $B$ must be nonnegative, from Lemma \ref{lema L-S} it is required that $-b_2(\sigma)\leq0$ or, in a equivalent way, that $a+s\geq\dfrac{b}{\sqrt{n}}$.
Thus
\begin{equation}\label{desiq1}
    s\geq\max\bigg\{0,\dfrac{b}{\sqrt{n}}-a\bigg\}\textcolor{red}{.}
\end{equation}
Therefore, $\sigma$ is realizable by the matrix $A=B+sI\geq0$ if and only if $\lambda_1\geq(n-1)a+ns,$ which shows the result.
\end{proof}

\begin{remark}
Notice that the condition in (\ref{desiq1}) is also a necessary and sufficient condition for the existence of a circulant nonnegative matrix equals to the matrix below.
{\small
\begin{equation} \label{matrixc}
\frac{1}{n}
\left(
\begin{array}{cccccc}
 \lambda_1-(n-1)a & \lambda_1+a+\sqrt{n}b & \lambda_1+a-\sqrt{n}b & \ldots & \lambda_1+a+\sqrt{n}b & \lambda_1+a-\sqrt{n}b\\
    \lambda_1+a-\sqrt{n}b & \lambda_1-(n-1)a & \lambda_1+a+\sqrt{n}b & \ldots & \lambda_1+a-\sqrt{n}b & \lambda_1+a+\sqrt{n}b \\
    \vdots & \vdots & \vdots & \ddots & \vdots & \vdots\\
    \lambda_1+a+\sqrt{n}b & \lambda_1+a-\sqrt{n}b & \lambda_1+a+\sqrt{n}b & \ldots & \lambda_1+a-\sqrt{n}b & \lambda_1-(n-1)a
  \end{array}
\right)
\end{equation}
}
with spectrum $\sigma=\{\lambda_1,-a+bi,-a-ib,\ldots,-a+ib,-a-ib\}$.
\end{remark}
\begin{definition}
A block circulant matrix is a matrix in the following form
\begin{eqnarray}
\label{circbloques}
\left(
\begin{array}{cccc}
A_0 & A_1 & \ldots & A_{m-1}\\
A_{m-1} & A_0 & \ldots & A_{m-2}\\
\vdots & \vdots & \ddots & \vdots \\
A_1 & A_2 & \ldots & A_0
\end{array}
\right),
\end{eqnarray}

where $A_i$ are $n$-by-$n$ arbitrary matrices.
\end{definition}

The partitioned into blocks matrices have particular importance in many areas. We refer here for instance Engineering, see  \cite{Kang} where the authors study forced vibration of symmetric structures. They present a method to calculate the eigenvectors of these matrices and then from the discrete structure they establish relationships for continuum structures. After discussing the dynamics aspects of the structures they consider subjects from earthquake engineering  and spectral analysis from such structures. Moreover, block circulant matrices are used in coding theory. For instance, in \cite{Karlin} the author used the canonical form based in circulant matrices to found many good codes: quadratic residue codes and high quality group codes. Some LDPC codes can also be defined by a matrix partitioned into blocks where each block are circulant matrices \cite{Kenneth}. See more applications in coding theory in \cite{Georgious}, and the references therein.
More on circulant block matrices with the circulant or factor circulant structure was considered for instance in \cite{Baker, Chao,Smith, Trapp}.

\section{Eigenpairs for square matrices partitioned into circulant blocks}

In this section we present spectral results for matrices partitioned into blocks where each block is a circulant matrix of order $n$.
The next theorem is valid in an algebraic closed field $K$ of characteristic $0$. For intance, $K=\mathbb{C}$.

\begin{theorem}
\label{main}
Let $K$ be an algebraically closed field of characteristic $0$ and suppose
that $A=\left( A(i,j)\right) $ is an $mn$-by-$mn$ matrix partitioned into $n^2$ circulant blocks of $m$-by-$m$ matrices, where for $ 1\leq i,j\leq n,$
\begin{equation}
\label{mtcsa}
A=\left( A(i,j)\right), \mbox{\, with\,} \ A(i,j)=circ(a(i,j)),
\end{equation}
where
\begin{align*}
 a(i,j)&=(a_0(i,j),a_1(i,j),\ldots,a_{m-1}(i,j))^T,
\end{align*}
with $a_k(i,j)\in K,\ \text{for }1\leq i,j\leq n \quad \text{and\quad} k=0,1,\ldots,m-1.$
For $k=0,1,\ldots,m-1,$ if
\begin{eqnarray}
\label{ek}
\mathbf{e_k}^T&=&\left(1,\omega^{k},\omega^{2k},\ldots,\omega^{(m-1)k}\right),
\end{eqnarray}
where $\omega$ is as in (\ref{omega_root}),
then

\begin{eqnarray}
\label{unionofsets}
\sigma \left( A\right) =\bigcup_{k=0}^{m-1}\sigma \left( S_k\right),
\end{eqnarray}
where
\begin{align}
\label{mtcsk}
S_k&=\left( s_k(i,j)\right)_{1\leq i,j\leq n}, \ \text{with }
s_k(i,j)=\mathbf{e_k}^Ta(i,j) ,
\end{align}
for $k=0,1,\ldots,m-1 \ \text{and}\ 1\leq i,j\leq n.$
\end{theorem}

\begin{proof}
The vector $\mathbf{e_k} $ in (\ref{ek}), is
the $(k+1)$-th column of $\sqrt{n}F,$ where $F$ is the matrix in (\ref{fourier-m}). It is clear that (\ref{eigenvalues}) implies that \[A(i,j)\mathbf{e_k}=s_k(i,j)\mathbf{e_k}.\]
Let $(\beta,u)$ with $u=(u_1,u_2,\ldots,u_n)$ be an eigenpair of the matrix $S_k$ in (\ref{mtcsk}), and
for $j=1,2,\ldots,n,$ let $v_j(k)=u_j\mathbf{e_k}.$

Since
\begin{eqnarray*}
A(i,j)v_{j}(k)      &      =       &   A(i,j)u_{j}\mathbf{e_k}\\
                    &      =       &s_{k}(i,j)u_{j}\mathbf{e_k}\\
                    &      =       &s_{k}(i,j)v_{j}(k) \text{,}
\end{eqnarray*}
then, for every $i=1,\ldots ,n,$
\begin{eqnarray*}
\sum\limits_{j=1}^{n}A(i,j)v_{j}(k)     &  =    &\left(\sum\limits_{j=1}^{n}s_{k}(i,j)u_{j}\right) \mathbf{e_k}\\
                                        &  =    &\left( \beta u_{i}\right) \mathbf{e_k}\\
                                        &  =    &\beta \left(u_{i}\mathbf{e_k}\right)\\
                                        &  =    &\beta v_{i}(k)\text{.}
\end{eqnarray*}
Therefore, the block vector $v_{k}=u\otimes \mathbf{e_k}$,  where $\otimes$ stands for the Kronecker matrix product (see \cite{Johnson3, Zhan}) is an eigenvector of $A$ and then, $\left(\beta ,v_{k}\right) $ is an eigenpair of $A$.
Thus, the union in the right hand side in (\ref{unionofsets}) is a subset of $\sigma\left( A\right)$ which has the same cardinally of $\sigma\left( A\right)$. The equality will follow after proving the linear independence of the set
$$\Upsilon= \Upsilon_{0} \cup \Upsilon_{1} \cup \cdots \cup \Upsilon_{m-1},$$ where, for $k=0,\ldots,m-1$,
$\Upsilon_{k}=\{u^{T}\otimes \mathbf{e_k}^{T}:u\in \Theta_k \},$
and $\Theta_{k}$
is a basis formed with eigenvectors of $S_k$.
For $k=0,\ldots,m-1$ consider $\Theta_k=\{u_{k}(1)^{T},\ldots, u_{k}(n)^{T}\}.$ Moreover, for each $k=0,1,\ldots m-1$ define the following matrix by its columns:
\[
\partial \left( k\right) =%
\begin{pmatrix}
u_{k}\left( 1\right) ^{T} \otimes \mathbf{e}_{k}^{T} & \cdots  & u_{k}\left( n\right) ^{T}\otimes \mathbf{e}_{k}^{T}
\end{pmatrix}^{T},
\]
 and in an analogous way, consider $H$ (defined by its columns) as follows:

\[
H=
\begin{pmatrix}
\partial \left( m-1\right) &\partial \left( m-2\right)  & \cdots & \partial \left( 1\right)&\partial \left( 0\right)
\end{pmatrix}.
\]
For $k=0,1,\ldots,m-1 \quad \text{and}\quad j=1,2,\ldots, n $ let us set
\begin{eqnarray}
\label{coord}
u_k(j)=(x_k(1,j),x_k(2,j),\ldots, x_k(n,j)).
\end{eqnarray}

A null linear combination of the columns of $H$ takes the form $$\sum _{k=0  }^{  m-1}{ \sum _{\ell=1 }^{ n }{ \alpha\left(k,\ell\right)\left(u_{k}\left( \ell \right) ^{T} \otimes \mathbf{e_k}^{T}\right) }  }=0. $$
By a routine computation and taking into account (\ref{coord}) and the linear independence of the columns of the matrix $F$ in (\ref{fourier-m}), the last outlined fact implies that
\begin{eqnarray}
\label{sisyem}
\sum _{ \ell =1 }^{ n }{ \alpha (k,\ell ){ x }_{ k }(t,\ell ) } =0,
\end{eqnarray}
for all $t=1,\ldots,n \quad \text{and}\quad k=0,\ldots,m-1.$
Since, for each $k$,  (\ref{sisyem}) is a linear system with $n$ equations and $n$ unknowns (namely $\alpha (k,1),\ldots ,\alpha (k,n)$) with coefficients matrix $B_k,$ where each matrix $B_k=\left(x_{k}(u,v)\right)_{1\leq u,v \leq n},$ we obtain $\alpha(k,\ell)=0$, for $\ell=1,\ldots,n$ and $k=0,\ldots,m-1.$ Therefore the result follows.

\end{proof}

The next result is a direct consequence of (\ref{coefficients}).
\begin{corollary}
\label{choose}

Let $\ell = 1,\ldots, m$, and the matrix $S_{\ell}$ defined as in (\ref{mtcsk}). Consider $a(u,v):=(a_0(u,v),\ldots, a_{m-1}(u,v))^T$ and  $A(u,v)=circ(a(u,v))$. Then,

\begin{eqnarray*}
\begin{pmatrix}
a_k(1,1) & a_k(1,2)&\ldots& a_k(1,m)\\
\vdots&\vdots&\ddots&\vdots\\
a_k(m,1) & a_k(m,2)&\ldots& a_k(m,m)
\end{pmatrix}
&=&
\frac{1}{m}\sum\limits_{\ell=0}^{m-1} S _{\ell}\omega ^{-k\ell},
\end{eqnarray*}

for $k=0,1,\ldots ,m-1.$

\end{corollary}

\begin{remark}
\label{remark1}
As a consequence of the above results, the matrix $A$ in (\ref{mtcsa}) is nonnegative if and only if for $k=0,1,\ldots ,m-1,$ the matrix

\begin{eqnarray}
\label{summatrix}
L_k:=\frac{1}{m}\sum\limits_{\ell=0}^{m-1} S _{\ell}\omega ^{-k\ell}
\end{eqnarray}
is nonnegative and, in this case the matrix $S_{0}$ is nonnegative. This latter affirmation is consequence
of the following sum:
\begin{eqnarray*}
\sum\limits_{k=0}^{m-1}L_{k} &=&\frac{1}{m}\sum\limits_{k=0}^{m-1}\sum%
\limits_{\ell =0}^{m-1}S_{\ell }\omega ^{-k\ell } \\
&=&\frac{1}{m}\sum\limits_{\ell =0}^{m-1}S_{\ell
}\sum\limits_{k=0}^{m-1}\omega ^{-k\ell } \\
&=&S_{0}+\frac{1}{m}\sum\limits_{\ell
=1}^{m-1}\sum\limits_{k=0}^{m-1}\omega ^{-k\ell }.
\end{eqnarray*}
In fact, $S_0$ is nonnegative as $\sum\limits_{k=0}^{m-1}L_{k}\geq 0$ and
\[
\sum\limits_{\ell =1}^{m-1}\sum\limits_{k=0}^{m-1}\omega ^{-k\ell
}=\sum\limits_{\ell =1}^{m-1}\frac{1-\left( \omega ^{-\ell }\right) ^{m}}{
1-\omega }=0.
\]

\end{remark}

\begin{remark} \label{15}
Let $q_1,q_2,\ldots,q_n$ be the canonical vectors of $\mathbb{R}^{n}$. For $k=0,1\ldots m-1,$ $1\leq u,v\leq n$ let us denote $\rho_{k}(u,v)=q_{u}^{T}S_{k}q_{v}$. From Corollary \ref{choose}, it is clear that
\begin{eqnarray}
\label{especific}
a_k(u,v)=\sum\limits_{\ell=0}^{m-1}\frac{1}{m} \rho_{\ell}(u,v)\omega ^{-k\ell}.
\end{eqnarray}
\end{remark}

The next result will be important in order to present a constructive criterion.

\begin{theorem}
\label{defLk}
For $k=0,1,\ldots, m-1$, let $L_k$ be the matrix defined by (\ref{summatrix}) and consider $A(u,v)$ as the $(u,v)-th$ circulant block matrix of the matrix $A$ in (\ref{mtcsa}) with $1\leq u,v \leq m.$ Then
\begin{eqnarray*}
A(u,v)&=&\left(circ(L_0^{T},L_{m-1}^{T},L_{m-2}^{T},\ldots,L_{1}^{T})\otimes q_u\right)^{T} \otimes q_v\\
&=&circ(q_u^{T}L_0^{T}q_{v},q_u^{T}L_1^{T}q_{v},\ldots,q_u^{T}L_{m-1}^{T}q_{v})
\end{eqnarray*}
\end{theorem}

\begin{proof}
From equation (\ref{especific}) in Remark \ref{15},  it is obtained the equality:
\begin{eqnarray*}
\label{especific2}
a_k(u,v)=q_{u}^{T}L_{k}q_{v}.
\end{eqnarray*}
Then,
\begin{eqnarray*}
\label{especific22}
A(u,v)&=&circ(a_0(u,v),a_1(u,v),\ldots,a_{m-1}(u,v))\\
&=&\left(circ(\left(L_{0}^{T}q_{u}\right)^{T},\left(L_{1}^{T}q_{u}\right)^{T},\ldots,\left(L_{m-1}^{T}q_{u}\right)^{T})\right)\otimes q_v\\
&=&\left(circ(L_{0}^{T}q_{u},L_{1}^{T}q_{u},\ldots,L_{m-1}^{T}q_{u})\right)^{T}\otimes q_v\\
&=&\left(circ(L_{0}^{T},L_{m-1}^{T},\ldots,L_{1}^{T})\otimes q_u\right)^{T}\otimes q_v.
\end{eqnarray*}
Therefore the statement is verified.
\end{proof}
From Theorem \ref{main} and Theorem \ref{defLk} the proof of the following result is clear.

\begin{corollary}
\label{converse2}
Let $\left( S_{\ell }^{'}\right) _{\ell =0}^{m-1}$ be $m$ $n$-by-$n$ complex matrices. Then there exists a matrix $A^{'}$ partitioned into blocks where each block is circulant and whose spectrum is given by
\begin{eqnarray}
\label{unionofsets2}
\sigma \left( A^{'}\right) =\bigcup_{\ell=0}^{m-1}\sigma \left( S_{\ell}^{'}\right),
\end{eqnarray}
\end{corollary}

\begin{proof}
For $k=0,1, \ldots,m-1$ let us consider the matrix $L_{k}^{'}$ defined as in (\ref{summatrix}) by using $S_{k}^{'}$ instead of $S_{k}$. Define the block matrix $A^{'}$ as in (\ref{mtcsa}) with the circulant blocks defined as in Theorem \ref{defLk} by using $L_{k}^{'}$ instead of $L_{k}$. From Theorem \ref{main} the spectrum of $A^{'}$ is given by the union in (\ref{unionofsets}). Moreover, from Theorem \ref{defLk} the $(u,v)$-th circulant block $A^{'}(u,v)$ of $A^{'}$ is obtained.  Since, this last circulant block and the original circulant block in the $(u,v)$-th position coincide, the union referred in (\ref{unionofsets}) and in (\ref{unionofsets2}) coincide. Thus the matrix $A^{'}$ has spectrum equal to the set in (\ref{unionofsets2}).
\end{proof}

The following example shows that for a given list, there exists a better splitting of this list such that there exists a matrix $A$ that realizes it.
\begin{example}
Let $1, \omega_1, \omega_2$ be the three complex cubic root of the unity. This means $\omega_1=\omega$ and $\omega_2=\omega^{2},$ where $\omega$ is as in (\ref{omega_root}) with $m=3.$ Consider the list $\left\{ 4,-3,\frac{1}{2}+i,\frac{1}{2}-i,\frac{1}{2}+i,\frac{1}{2}-i\right\}$ and choose the partition of the given set as follows:
\[
\left\{ 4,-3\right\} \cup \left\{ \frac{1}{2}+i,\frac{1}{2}-i\right\} \cup
\left\{ \frac{1}{2}+i,\frac{1}{2}-i\right\}.
\]
Therefore, according to this partition it is possible to obtain matrices
\begin{eqnarray*}
S_{0} &=&
         \begin{pmatrix}
          \frac{1}{2} & \frac{7}{2} \\
          \frac{7}{2} & \frac{1}{2}
         \end{pmatrix}
, \text{with eigenvalues }\, 4,-3, \\
S_{1} &=&
        \begin{pmatrix}
          \frac{1}{2} & -1 \\
           1          & \frac{1}{2}
        \end{pmatrix}
,\text{with eigenvalues }\, \tfrac{1}{2}+i,\tfrac{1}{2}-i, \\
S_{2} &=&
\begin{pmatrix}
      \frac{1}{2} & -1 \\
                1 & \frac{1}{2}
\end{pmatrix}%
,\text{with eigenvalues }\, \tfrac{1}{2}+i,\tfrac{1}{2}-i.
\end{eqnarray*}

Thus, as $\omega_{1}=\frac{-1+i\sqrt{3}}{2}$ and $\omega_{2}=\frac{-1-i\sqrt{3}}{2},$ we have:

\begin{eqnarray*}
3L_{0} &=&S_{0}+S_{1}+S_{2} =
\begin{pmatrix}
\frac{3}{2} & \frac{3}{2} \\
\frac{11}{2} & \frac{3}{2}%
\end{pmatrix},
\\
3L_{1} &=&S_{0}+\omega_2S_{1}+\omega_1 S_{2} =
\begin{pmatrix}
0 & \frac{9}{2} \\
\frac{5}{2} & 0%
\end{pmatrix},
\\
3L_{2} &=&S_{0}+\omega_1 S_{1}+\omega_2S_{2} =
\begin{pmatrix}
0 & \frac{9}{2} \\
\frac{5}{2} & 0%
\end{pmatrix}.
\end{eqnarray*}
Then, the circulant blocks are:
\begin{align*}
A(1,1)&=circ\left(\frac{1}{2},0,0\right),\
A(1,2)=circ\left(\frac{1}{2},\frac{3}{2},\frac{3}{2}\right),\\
A(2,2)&=circ\left(\frac{1}{2},0,0\right),\
A(2,1)=circ\left(\frac{11}{6},\frac{5}{6},\frac{5}{6}\right).\\
\end{align*}
Therefore, by Theorem \ref{main} we obtain the nonnegative matrix

\[
A=%
\begin{pmatrix}
\frac{1}{2} & 0 & 0 & \frac{1}{2} & \frac{3}{2} & \frac{3}{2} \\
0 & \frac{1}{2} & 0 & \frac{3}{2} & \frac{1}{2} & \frac{3}{2} \\
0 & 0 & \frac{1}{2} & \frac{3}{2} & \frac{3}{2} & \frac{1}{2} \\
\frac{11}{6} & \frac{5}{6} & \frac{5}{6} & \frac{1}{2} & 0 & 0 \\
\frac{5}{6} & \frac{11}{6} & \frac{5}{6} & 0 & \frac{1}{2} & 0 \\
\frac{5}{6} & \frac{5}{6} & \frac{11}{6} & 0 & 0 & \frac{1}{2}%
\end{pmatrix}
\]\\
whose spectrum is $\left\{ 4,-3,\frac{1}{2}+i,\frac{1}{2}-i,\frac{1}{2}+i,\frac{1}{2}-i\right\}.$ Nevertheless, considering another partition of the initial set

\[ \left\{ 4,\frac{1}{2}+i,\frac{1%
}{2}-i\right\} \cup \left\{ -3,\frac{1}{2}+i,\frac{1}{2}-i\right\} ,\]
by the above procedure we obtain,
\[
{\small
 S_{1}=\frac{1}{3}
\left(\begin{array}{ccc}
-2 & \sqrt{3}-\frac{7}{2} & -\sqrt{3}-\frac{7}{2} \\
-\sqrt{3}-\frac{7}{2} & -2 & \sqrt{3}-\frac{7}{2} \\
\sqrt{3}-\frac{7}{2} & -\sqrt{3}-\frac{7}{2} & -2
\end{array}
\right)
=\allowbreak
\left(\begin{array}{ccc}
-\frac{2}{3} & \frac{\sqrt{3}}{3}-\frac{7}{6} & -\frac{\sqrt{3}}{3}-\frac{7%
}{6} \\
-\frac{\sqrt{3}}{3}-\frac{7}{6} & -\frac{2}{3} & \frac{\sqrt{3}}{3}-\frac{7%
}{6} \\
\frac{\sqrt{3}}{3}-\frac{7}{6} & -\frac{\sqrt{3}}{3}-\frac{7}{6} & -\frac{2%
}{3}
\end{array}
\right)
}\]
with eigenvalues $ \frac{1}{2}+i,\frac{1}{2}-i,-3,$ and
\[
{\small S_{0}=\frac{1}{3}%
\left(\begin{array}{ccc}
5 & \sqrt{3}+\frac{7}{2} & -\sqrt{3}+\frac{7}{2} \\
-\sqrt{3}+\frac{7}{2} & 5 & \sqrt{3}+\frac{7}{2} \\
\sqrt{3}+\frac{7}{2} & -\sqrt{3}+\frac{7}{2} & 5%
\end{array}
\right)
=\allowbreak
\left(\begin{array}{ccc}
\frac{5}{3} & \frac{\sqrt{3}}{3}+\frac{7}{6} & \frac{7}{6}-\frac{\sqrt{3}}{3} \\
\frac{7}{6}-\frac{\sqrt{3}}{3} & \frac{5}{3} & \frac{\sqrt{3}}{3}+\frac{7}{%
6} \\
\frac{\sqrt{3}}{3}+\frac{7}{6} & \frac{7}{6}-\frac{\sqrt{3}}{3}& \frac{5}{%
3}
\end{array}
\right)}\]
with eigenvalues $\frac{1}{2}+i,\frac{1}{2}-i,4.$ Then

\[S_{0}+S_{1}
=\allowbreak
\small{\begin{pmatrix}
1 & \frac{2\sqrt{3}}{3} & -\frac{2\sqrt{3}}{3} \\
-\frac{2\sqrt{3}}{3} & 1 & \frac{2\sqrt{3}}{3} \\
\frac{2\sqrt{3}}{3} & -\frac{2\sqrt{3}}{3} & 1%
\end{pmatrix}}
\]
\\
which is not a nonnegative matrix. Therefore, using the previous procedure, the second partition does not allow us to obtain a nonnegative matrix with the given spectrum.
\end{example}

\section{On structured matrices partitioned into circulant blocks matrices}

In this section we search conditions to obtain a given structure on a partitioned into blocks matrix $A$ with circulant blocks as in (\ref{mtcsa}) using the structure of the matrices $S_{k}$ in (\ref{mtcsk}).
Here, a matrix partitioned into blocks is called \textit{block permutative matrix} when all its row blocks ( up to the first one) are permutations of precisely its first row block.

\begin{theorem}
Let $A$ be the matrix partitioned into blocks as defined in (\ref{mtcsa}). Let $S_k$ be the class of matrices defined in (\ref{mtcsk}), for $k=1,2,\ldots,m-1.$
Then
\begin{enumerate}
    \item If $S_k$ is diagonal for $k=0,1,\ldots,m-1$, then the matrix $A$ is a diagonal block matrix with circulant blocks.
    \item The matrix $A$ is a block circulant matrix, (see (\ref{circbloques})), with circulant blocks if and only if $S_k$ is a circulant matrix for all $k=0,1,\ldots,m-1.$
    \item The matrix $A$ is a block permutative matrix with circulant blocks
    if and only if $S_k$ are permutatively equivalent matrices, for all $k=0,1,\ldots,m-1.$
    \item The matrix $A$ is a symmetric real matrix partitioned into circulant blocks
      if and only if $S_k$ is a real and symmetric matrix for all $k=0,1,\ldots,m-1.$
\end{enumerate}
\end{theorem}

\begin{proof}
Suppose that for all $\ell =0,1,\ldots ,m-1$ the matrices $%
S_{\ell }$ are diagonal. We will prove that $A=\left( A\left( u,v\right)
\right) $ is a diagonal block matrix with circulant blocks. It is clear that for $%
u \neq v$, $\ q_{u}^{T}S_{\ell }q_{v}=0$ therefore, for all $k=0,1,\ldots.m-1$
\[
q_{u}^{T}L_{k}q_{v}=\tfrac{1}{m}\sum\limits_{\ell =0}^{m-1}\omega ^{-k\ell
}q_{u}^{T}S_{\ell }q_{v}=0.
\]%
Then, for $u\neq v,$  by Theorem \ref{defLk},  $A\left( u,v\right) =0.$  Thus  $A$ is a diagonal block matrix with circulant blocks.

Suppose that $A$ is a block circulant matrix with circulant blocks, since the entries
of $S_{k}$ \ follow the distribution of the blocks of $A$ the matrices
$S_{k}$ are circulant.
Conversely, assume that for all $k=0,1,\ldots ,m-1,\ S_{k}$ is circulant,
then for $k=1,2,\ldots,m$ the matrix $L_k$ defined in (\ref{summatrix})
is also circulant. Let us suppose that
\begin{equation*}
L_{k}=circ\left(\varrho _{0k}, \varrho _{1k}, \ldots,   \varrho _{(n-1)k}\right),
\end{equation*}%
then%
\begin{equation*}
q_{u}^{T}L_{k}q_{v}=\left\{
\begin{tabular}{ll}
$\varrho _{(v-u)k}$ & $1\leq u\leq v\leq n,$ \\
$\varrho _{(n-u+v)k}$ & $1\leq v<u\leq n;$%
\end{tabular}%
\right.
\end{equation*}%
since
\begin{eqnarray*}
A\left(u,v\right)  &= &circ(q_{u}^{T}L_{0}q_{v},q_{u}^{T}L_{1}q_{v},\ldots
,q_{u}^{T}L_{m-1}q_{v}) \\
&=&\left \{
\begin{tabular}{ll}
$circ\left( \varrho _{(v-u)0},\varrho _{(v-u)1},\ldots ,\varrho
_{(v-u)(m-1)}\right) $ & $1\leq u\leq v\leq n,$ \\
$circ\left( \varrho _{(n-u+v)0},\varrho _{(n-u+v)1},\ldots ,\varrho
_{(n-u+v)(m-1)}\right) $ & $1\leq v<u\leq n.$%
\end{tabular}
\right.
\end{eqnarray*}
Thus the matrix $A=(A(u,v))$ partitioned into blocks is block circulant.

Let us suppose now that $A$ is a block permutative matrix. Then there exits a
permutation vector
\begin{equation}
\nu =\left( \nu _{0},\nu _{1},\ldots ,\nu _{n-1}\right)   \label{vectper}
\end{equation}%
with $\nu _{0}=id$, such that
\begin{equation*}
A=%
\begin{pmatrix}
A\left( 1,1\right)  & A\left( 1,2\right)  & \ldots  & A\left( 1,n\right)  \\
A\left( 1,\nu _{1}\left( 1\right) \right)  & A\left( 1,\nu _{1}\left(
2\right) \right)  & \ldots  & A\left( 1,\nu _{1}\left( n\right) \right)  \\
\vdots  & \vdots  & \ddots  & \vdots \\
A\left( 1,\nu _{n-1}\left( 1\right) \right)  & \ldots  & \ldots & A\left( 1,\nu
_{n-1}\left( n\right) \right)
\end{pmatrix}.
\end{equation*}%
Thus
\begin{equation*}
S_{k}=\left( s_{k}\left( i,j\right) \right) =%
\begin{pmatrix}
s_{k}\left( 1,1\right)  & s_{k}\left( 1,2\right)  & \ldots  & s_{k}\left(
1,n\right)  \\
s_{k}\left( 1,\nu _{1}\left( 1\right) \right)  & s_{k}\left( 1,\nu
_{1}\left( 2\right) \right)  & \ldots  & s_{k}\left( 1,\nu _{1}\left(
n\right) \right)  \\
\vdots  & \vdots  & \ddots  & \vdots \\
s_{k}\left( 1,\nu _{n-1}\left( 1\right) \right)  & \ldots  & \ldots & s_{k}\left(
1,\nu _{n-1}\left( n\right) \right)
\end{pmatrix}%
\end{equation*}%
this means that the set of matrices $S_{k}$ are permutatively equivalent.
Conversely, if $S_{k}$ are permutatively equivalent matrices for all $%
k=0,1,\ldots ,m-1,\ $ then the matrix $L_k$ defined in (\ref{summatrix}) is also permutative. Let us suppose that there exist a vector $b=\left(
b_{1k},b_{2k},\ldots ,b_{nk}\right) $ and a permutation vector as in (\ref%
{vectper}) such that

\begin{equation*}
L_{k}=
\begin{pmatrix}
b_{1,k} & b_{2,k} & \ldots  & b_{n,k} \\
b_{\nu _{1}( 1) },_{k} & b_{\nu _{1}(2)},_{k} &
\ldots  & b_{\nu _{1}\left( n\right) },_{k} \\
\vdots  & \vdots  & \ddots  & \vdots \\
b_{\nu _{n-1}\left( 1\right) },_{k} & b_{\nu _{n-1}\left( 2\right) },_{k} & \ldots  & b_{\nu _{n-1}\left( n\right) },_{k}%
\end{pmatrix}.
\end{equation*}
Then
\begin{equation*}
q_{u}^{T}L_{k}q_{v}=b_{\nu _{u-1}\left( v\right) },_{k},
\end{equation*}
since
\begin{eqnarray*}
A\left( u,v\right)  &=&circ(q_{u}^{T}L_{0}q_{v},q_{u}^{T}L_{1}q_{v},\ldots
,q_{u}^{T}L_{m-1}q_{v}) \\
&=&circ\left( b_{\nu _{u-1}\left( v\right) },_{0},b_{\nu _{u-1}\left(
v\right) },_{1},\ldots ,b_{\nu _{u-1}\left( v\right) },_{m-1}\right).
\end{eqnarray*}


From now on, in order to simplify the notation and unless we say the contrary, it is written $\left( a_{0},a_{1},\ldots ,a_{m-1}\right)$ instead of  $circ\left( a_{0},a_{1},\ldots ,a_{m-1}\right). $ Thus, $A$ can be constructed as follows:

\begin{equation*}
{\tiny \begin{pmatrix}
\left( b_{1,0},\ldots ,b_{1,m-1}\right)  & \left(
b_{2,0},\ldots ,b_{2,m-1}\right)  & \ldots  & \left(
b_{n,0},\ldots ,b_{n,m-1}\right)  \\
\left( b_{\nu _{1}\left( 1\right) },_{0},\ldots ,b_{\nu _{1}\left( 1\right) },_{m-1}\right)  & \left( b_{\nu
_{1}\left( 2\right) },_{0},\ldots ,b_{\nu
_{1}\left( 2\right) },_{m-1}\right)  & \ldots  & \vdots  \\
\vdots  & \vdots & \ddots  & \vdots \\
\left( b_{\nu _{n-1}\left( 1\right) },_{0},\ldots ,b_{\nu _{n-1}\left( 1\right) },_{m-1}\right)  & \ldots & \ldots&
\left( b_{\nu _{n-1}\left( n\right) },_{0},\ldots ,b_{\nu _{n-1}\left( n\right) },_{m-1}\right)
\end{pmatrix}}
\end{equation*}
Therefore, $A$ is a block permutative matrix with circulant blocks.\\
In order to prove Item 4., suppose that $A$ is symmetric partitioned into circulant block matrices.
Since the entries
of $S_{k}$ \ follow the distribution of the blocks of $A$ then $S_{k}$ is symmetric, for all $k=0,1,\ldots ,m-1.$
Conversely, assume that for all $k=0,1,\ldots ,m-1,\ S_{k}$ is a symmetric
real matrix. Since for all $k=0,1,\ldots ,m-1,\ s_{k}\left( u,v\right) $ is
an eigenvalue of $A\left( u,v\right) $, then the circulant matrices $A\left(
u,v\right) $ have only real eigenvalues, so they should be symmetric (see in
\cite{RSGuo}). Moreover,  for all $k=0,1,\ldots ,m-1$, the $L_k$ is the matrix defined in (\ref{summatrix}) and is symmetric. Suppose that
\begin{equation*}
L_{k}=%
\begin{pmatrix}
\varrho _{k}\left( 1,1\right)  & \varrho _{k}\left( 1,2\right)  & \ldots  &
\varrho _{k}\left( 1,n\right)  \\
\varrho _{k}\left( 1,2\right)  & \varrho _{k}\left( 2,2\right) & \ldots  &
\varrho _{k}\left( 2,n\right)  \\
\vdots  & \vdots  & \ddots  & \vdots\\
\varrho _{k}\left( 1,n\right)  & \ldots  & \ldots & \varrho _{k}\left( n,n\right)
\end{pmatrix}%
\end{equation*}%
then%
\begin{eqnarray*}
A\left( v,u\right)  &=&circ(q_{v}^{T}L_{0}q_{u},q_{v}^{T}L_{1}q_{u},\ldots
,q_{v}^{T}L_{m-1}q_{u}) \\
&=&circ\left( \varrho _{0}\left( v,u\right) ,\varrho _{1}\left( v,u\right)
,\ldots ,\varrho _{m-1}\left( v,u\right) \right)  \\
&=&circ\left( \varrho _{0}\left( u,v\right) ,\varrho _{1}\left( u,v\right)
,\ldots ,\varrho _{m-1}\left( u,v\right) \right)  \\
&=&A\left( u,v\right) .
\end{eqnarray*}%
Therefore,
\[
A^{T} =\left( A\left( u,v\right) \right)^{T}=
\left( A\left( v,u\right)^{T} \right)=
\left( A\left( v,u\right) \right)=
A.
\]

Thus, the matrix A is symmetric.
\end{proof}

\section{An inverse problem related to block circulant matrices with circulant blocks}

In this section we study the Guo index for some structured matrices. Namely, we dedicate our attention to block circulant matrices with circulant blocks, which are a type of permutative matrices. To this purpose we  rise to the following inverse problem.
\begin{problem}
\label{GUO2}
Let $q_1,q_2,\ldots,q_m$ be as in Remark \ref{15} and $E=(\varepsilon_{ij})$ be an $n$-by-$m$ matrix where the multiset $\left\{ E\right\} $ and is formed by the entries of $E$ is closed under complex conjugation. Suppose that $\varepsilon_{11}$ is positive and has the largest absolute value among the absolute values of entries of $E$, the entries of the first column of $E$ are the eigenvalues of a nonnegative circulant matrix $S_0$, and for $\ell=1,\ldots,\left\lfloor \tfrac{m}{2}\right\rfloor$,
\begin{eqnarray}
\label{conjugatecondition3}
Eq_{\ell+1}=\overline{Eq_{m-\ell+1}}
\end{eqnarray}
that is, the entries of the $(m-\ell+1)$-th column of $E$ are the corresponding complex conjugate of the $(\ell+1)$-th column of $E$, for $\ell=1,\ldots,\left \lfloor \tfrac{m}{2} \right \rfloor.$ Note that for $m=2h$ the condition in (\ref{conjugatecondition3}) implies that the column $Eq_{h+1}$ of $E$ has real entries. Which condition (or conditions) is (or are) sufficient for the existence of a nonnegative matrix $A$ partitioned into blocks where each block is a circulant matrix whose spectrum equals to $\left \{ E \right \}$?
\end{problem}
In order to give a response to this question one needs to introduce the following DFT matrix

\begin{equation*}
G =\frac{1}{\sqrt{n}}%
\begin{pmatrix}
1 & 1 & 1 & \ldots  & 1 & 1 \\
1 & \tau  & \tau ^{2} & \ldots  & \tau ^{n-2} & \tau ^{n-1} \\
1 & \tau ^{2} & \tau ^{4} & \ldots  & \vdots & \vdots  \\
\vdots  & \vdots  & \vdots  & \ddots  & \vdots  & \vdots \\
1 & \tau ^{n-1} & \tau ^{2\left( n-1\right) } & \ldots  & \ldots & \tau ^{\left(
n-1\right) ^{2}}%
\end{pmatrix}%
.
\end{equation*}
where
\begin{equation}
\tau =\cos\left(\frac{2\pi }{n}\right)+i\sin \left(\frac{2\pi }{n}\right).  \label{tau}
\end{equation}
Moreover, it is necessary to define \textit{circulant real list}.
\begin{definition}
Given the list $\Lambda =\left( \lambda_0 ,\lambda _{1},\lambda
_{2},\ldots ,\lambda _{n-1}\right) $ we say that $%
\Lambda $ is a circulant real list if the following conditions hold:
\begin{enumerate}
\item $\lambda _{0}=\rho=\max \left \{ \left \vert \lambda _{j}\right\vert
:j=1,2,\ldots ,n\right\}$ and
\item $\lambda _{n-k}=\overline{\lambda _{k}}$, for $k=1,2,\ldots ,n-1.$
\end{enumerate}
\end{definition}

The following result gives the Guo's Index for circulant matrices.
\begin{theorem}\cite[Theorem 4]{RSGuo} \label{conditions}
Let $\Lambda =\left( \lambda _{0},\lambda _{1},\ldots ,\lambda
_{n-1}\right) $ be a circulant real list and consider the set
\begin{equation*}
\small{\mathcal{P}=\left\{ \alpha \in \mathbb{S}_{n}\text{:\ }\alpha \left( 0\right) =0 \ \text{and} \ \alpha
\left( n-k\right) =n-\alpha \left( k\right),k\in \{1,2,\ldots,n-1\}\right\} . \label{p}}
\end{equation*}
then necessary and
sufficient conditions for $\Lambda $ to be the spectrum of a real circulant
matrix are given by (\ref{ns1}) and (\ref{ns2}) as follows
\begin{equation}
\lambda _{0}\geq \min_{\alpha \in \mathcal{P}}\max_{0\leq k\leq 2m}
\left\{
\begin{array}{c}
-2\sum\limits_{j=1}^{m}\rm{Re} \lambda _{\alpha\left(j\right)}\cos \frac{2k j\pi }{2m+1}- \\
\qquad -2\sum\limits_{j=1}^{m}\rm{Im}\lambda _{\alpha\left(j\right)} \sin\frac{2k j\pi }{2m+1}
\end{array}
\label{ns1}
\right.
\end{equation}
whenever $n=2m-1$, and
\begin{equation}
\lambda _{0}\geq \min_{\alpha \in \mathcal{P}}\max_{0\leq k\leq 2m+1}
\left\{
\begin{array}{c}
-2\sum\limits_{j=1}^{m-1}\rm{Re}\lambda _{\alpha\left(j\right)}\cos \frac{
2k j\pi }{m+1}-\left( -1\right) ^{k}\lambda _{m}- \\
\qquad -2\sum\limits_{j=1}^{m-1}\rm{Im}\lambda _{\alpha\left(j\right)}
\sin \frac{2k j\pi }{m+1}%
\end{array}%
  \label{ns2}
  \right.
\end{equation}
whenever $n=2m-2$.
\end{theorem}

When the matrix is block circulant with circulant blocks we have the following result.

\begin{theorem}
\label{strongguo}
Let $E=(\varepsilon_{ij})$ be an $n$-by-$m$ matrix such that the set formed by its first column is the spectrum of a nonnegative circulant matrix $S_0.$ The multiset $\left\{ E\right\} $ formed by the entries of $E$ is closed under complex conjugation. $\varepsilon_{11}$ is positive, and has the largest absolute value among the absolute values of the entries of $E$, and for $\ell=1,\ldots,\left \lfloor \tfrac{m}{2} \right \rfloor$
\begin{eqnarray}
\label{evenconjugate}
Eq_{\ell+1}=\overline{Eq_{m-\ell+1}}.
\end{eqnarray}
If
\begin{eqnarray}
\label{new1}
\varepsilon_{11}\geq \Phi
\end{eqnarray}
with \[\Phi=\max_{k\in \{0, \ldots, m\}}
\left(\sum\limits_{p=1}^{n-1}\varepsilon_{(p+1)1}\tau^{-jp}+\sum_{\ell =1}^{m-1}\omega ^{-k\ell}\varepsilon_{1(\ell+1)}+\sum_{\ell =1}^{m-1}\sum\limits_{p=1}^{n-1}\varepsilon_{(p+1)(\ell+1)}\omega ^{-k\ell}\tau^{-jp}              \right)\]
for all $k=0, 1, \ldots, m-1,$
then $\left \{ E \right \}$ is the spectrum of a nonnegative block circulant matrix $A$ whose blocks are circulant.
\end{theorem}

\begin{proof}
By the conditions of the statement there exists a nonnegative circulant matrix \[S_0:=circ\left(
s_{00},s_{10},\ldots ,s_{(n-1)0}\right) \] whose spectrum is $\left \{ Eq_{1} \right \}$ (the set of the entries in $Eq_1$).
The condition in (\ref{evenconjugate}) implies that for $\ell=1,\ldots,\left \lfloor \frac{m}{2} \right \rfloor$ the circulant matrices $S_{\ell+1}$ and $S_{m-\ell}$ whose spectrum are $\left \{ Eq_{(\ell+1)} \right \}$ and $\left \{ Eq_{(m-\ell+1)} \right \}$, respectively, are related by  $S_{\ell+1}^{\ast}=S_{m-\ell+1}$. For $\ell=1,\ldots, m-1$, suppose that
\begin{equation*}
S_{\ell}=circ\left( s\left( \ell\right) \right) \text{, with }s\left( \ell\right)
=\left( s_{0\ell},s_{1\ell},\ldots ,s_{(n-1)\ell}\right) ,
\end{equation*}%
where
\begin{equation}
s\left( \ell\right)^{T} =\frac{1}{\sqrt{n}}G^{\ast } Eq_{\ell} \label{circsl}.
\end{equation}

The entries of the $n$-by-$n$ nonnegative circulant matrices can be obtained, using equation (\ref{coefficients}) and the entries of the sums
\begin{equation}
L_k=\frac{1}{m}S_0+\frac{1}{m}\sum_{\ell =1}^{m-1}S_{\ell }\omega ^{-k\ell
}, \label{circlkp1}
\end{equation}
$k=0,1,\ldots, m-1.$
Recalling that the linear combination of circulant matrices are circulant, \cite{MAR}, then the matrices $L_k$ in (\ref{circlkp1})  are circulant. Suppose that $$L_k=circ\left(a_0(k),\ldots, a_{n-1}(k)\right).$$
From (\ref{circlkp1}), for $j=0,1,\ldots,n-1$ the following
holds:

\begin{eqnarray*}
a_j(k) &=&\frac{1}{m}\left( s_{j0}+\sum_{\ell =1}^{m-1}\omega ^{-k\ell
}s_{j\ell } \right).
\end{eqnarray*}
Using (\ref{circsl}), we have
\begin{eqnarray*}
s_{j\ell}&=&\frac{1}{n}[\varepsilon_{1(\ell+1)}+\sum\limits_{p=1}^{n-1}\varepsilon_{(p+1)(\ell+1)}\tau^{-jp}],
\end{eqnarray*}
for $j=0,1,\ldots,n-1.$

On the other hand, from the expression of $a_j(k)$ and its nonnegativity condition we can write:
\small{
\begin{eqnarray*}
0 \leq a_j(k)&=&\frac{1}{m}\left( s_{j0}+\sum_{\ell =1}^{m-1}\omega ^{-k\ell
}s_{j\ell } \right)\\
             & = &\small{\frac{1}{m}}\frac{1}{n}\left(
               [\varepsilon_{11}+\sum\limits_{p=1}^{n-1}\varepsilon_{(p+1)1}\tau^{-jp}]+\sum_{\ell =1}^{m-1}\omega ^{-k\ell}[\varepsilon_{1(\ell+1)}+\sum\limits_{p=1}^{n-1}\varepsilon_{(p+1)(\ell+1)}\tau^{-jp}]
               \right) \\
             &=&\small{\frac{1}{mn}}\left(\varepsilon_{11}+\sum\limits_{p=1}^{n-1}\varepsilon_{(p+1)1}\tau^{-jp}+\sum_{\ell =1}^{m-1}\omega ^{-k\ell}\varepsilon_{1(\ell+1)}+\sum_{\ell =1}^{m-1}\sum\limits_{p=1}^{n-1}\varepsilon_{(p+1)(\ell+1)}\omega ^{-k\ell}\tau^{-jp}
               \right)
\end{eqnarray*}
}
for all $k=0,1,\ldots, m-1.$
Therefore, the last condition implies the inequality in (\ref{new1}).
\end{proof}
Now, one can formulate the following question.
Under which conditions  the multiset $\left \{ E \right \}$ formed with the entries of an $n$-by-$m$ matrix $E=\left( \varepsilon _{ij}\right)$ as in Problem \ref{GUO2} is the spectrum of a nonnegative block matrix with circulant blocks.
 Let us consider the set
\[
\boldsymbol{P}=\left\{ f:\left\{ E\right\} \rightarrow \left\{ E\right\} :f%
\text{ }is\ bijective\right\}
\]%
\begin{definition}
The function $f\in \boldsymbol{P}$ is said to be \textit{$E$-nonnegative spectrally stable ($E$-NNSS)} if
the matrix $E\left( f\right) =\left( f\left( \varepsilon _{ij}\right)
\right) $ has the same characteristic of the matrix $E$.
That is equivalent to say that $
f\left( \varepsilon _{11}\right) $ has the maximum absolute value among $%
f\left( \varepsilon _{ij}\right), $ the first column of $E\left( f\right) $
corresponds to the set of eigenvalues of a nonnegative circulant matrix, the
$\left( m-\ell +1 \right) $-th column of $E\left( f\right) $ is the complex
conjugate column of the $\left( \ell +1\right) $-th column  of  $E\left(
f\right) $, and then
\[
E\left( f\right) q_{(m-\ell +1) }=\overline{E\left( f\right) }q_{\left( \ell
+1\right) },
\]%
for $\ell=1,\ldots,\left \lfloor \frac{m}{2}  \right \rfloor.$
\end{definition}

For instance:

\begin{enumerate}

\item The identity function of $E$ into $E$ is clearly $E$-NNSS.
\item   $f:\left\{ E\right\} \rightarrow \left\{ E\right\} $ defined by
\[
f\left( \varepsilon _{ij}\right) =\left \{
\begin{array}{ll}
\varepsilon _{ij} & j\neq 2 \quad {\small \text{and}} \quad j\neq m, \\
\varepsilon_{i m } & j=2, \\
\varepsilon _{i2} & j=m;
\end{array}
\right.
\]
is $E$-NNSS.
\item  $f:\left\{ E\right\} \rightarrow \left\{ E\right\} $ defined by
\[
f\left( \varepsilon _{ij}\right) =\overline{\varepsilon }_{ij}\mbox{\, for all\,} i,j
\]
is $E$-NNSS.
\item  $f:\left\{ E\right\} \rightarrow \left\{ E\right\} $ defined by
\[
f\left( \varepsilon_{ij}\right) =
\left\{
\begin{array}{ll}
\overline{\varepsilon}_{ij} & \mbox{\, if \,} j=1,\\
\varepsilon_{ij} & \mbox{\,if \,} j\neq 1;
\end{array}
\right.
\]
is $E$-NNSS.
\end{enumerate}
We denote by $\boldsymbol{P}^{*}$ the subset of $\boldsymbol{P}$ formed by all $E$-NNSS bijections of $E$.
Note that if $f \in P^{*}$ then $f(\varepsilon_{11})=\varepsilon_{11}.$

\begin{theorem}
Let $E=(\varepsilon_{ij})$ be an $n$-by-$m$ matrix such that the set formed by its first column is the spectrum of a nonnegative circulant matrix $S_0,$ the multiset $\left\{ E\right\} $ formed by the entries of $E$ is closed under complex conjugation, $\varepsilon_{11}$ is positive and has the largest absolute value among the absolute values of the entries of $E$, and for $\ell=1,\ldots,\left \lfloor \frac{m}{2} \right \rfloor$
\begin{eqnarray}
Eq_{\ell+1}=\overline{Eq_{m-\ell +1}}.
\end{eqnarray}
The multiset $\left \{ E \right \}$ is the spectrum of a nonnegative block circulant matrix $A$ whose blocks are circulant if and only if

\begin{eqnarray}
\label{maria}
\varepsilon_{11} &\geq &
\min_{f\in \boldsymbol{P}^{*}}
\max_{k\in \{0, \ldots, m\}} \Theta,
\end{eqnarray}
where,
\begin{eqnarray*}
\Theta &  = &
\sum\limits_{p=1}^{n-1}f(\varepsilon_{(p+1)1})\tau^{-jp}+\sum_{\ell =1}^{m-1}\omega ^{-k\ell}f(\varepsilon_{1(\ell+1)})+\sum_{\ell =1}^{m-1}\sum\limits_{p=1}^{n-1}f(\varepsilon_{(p+1)(\ell+1)})\omega ^{-k\ell}\tau^{-jp}
\end{eqnarray*}

\end{theorem}

\begin{proof}
Following the same steps of the above proof this time replacing $\varepsilon_{ij}$ by $f(\varepsilon_{ij})$ we arrive at the inequality in (\ref{new1}). After taking the minimum when the function $f$ vary into $\boldsymbol{P}^{*}$, the inequality (\ref{maria}) is obtained.
\end{proof}

\begin{example}
    For
    \[
    E_{4}=%
    \begin{pmatrix}
    4 & -1+i & -1-i \\
    -1 & i & -i%
    \end{pmatrix}%
    \]%
    the matrices
    \begin{eqnarray*}
    S_{0} &=&%
    \begin{pmatrix}
    1.5 & 2.5 \\
    2.5 & 1.5%
    \end{pmatrix}
    \\
    S_{1} &=&%
    \begin{pmatrix}
    -0.5+i & -0.5 \\
    -0.5 & -0.5+i%
    \end{pmatrix}
    \\
    S_{2} &=&%
    \begin{pmatrix}
    -0.5-i & -0.5 \\
    -0.5 & -0.5-i%
    \end{pmatrix}%
    \end{eqnarray*}%
    are obtained. Thus
    \begin{eqnarray*}
    L_{0} &=&%
    \begin{pmatrix}
    0.1667 & 0.5 \\
    0.5 & 0.1667%
    \end{pmatrix}
    \\
    L_{1} &=&%
    \begin{pmatrix}
    1.2440 & 1 \\
    1 & 1.2440%
    \end{pmatrix}
    \\
    L_{2} &=&%
    \begin{pmatrix}
    0.0893 & 1 \\
    1 & 0.0893%
    \end{pmatrix}%
    \end{eqnarray*}%
    can be constructed. Now by dimishing the spectral radius and considering
    \[
    E_{3}=%
    \begin{pmatrix}
    3 & -1+i & -1-i \\
    -1 & -i & i%
    \end{pmatrix}%
    \]%
    the matrices
    \begin{eqnarray*}
    S_{0} &=&%
    \begin{pmatrix}
    1 & 2 \\
    2 & 1%
    \end{pmatrix}
    \\
    S_{1} &=&%
    \begin{pmatrix}
    -0.5 & -0.5+i \\
    -0.5+i & -0.5%
    \end{pmatrix}
    \\
    S_{2} &=&%
    \begin{pmatrix}
    -0.5 & -0.5-i \\
    -0.5-i & -0.5%
    \end{pmatrix}%
    \end{eqnarray*}%
     are yield. Thus
    \begin{eqnarray*}
    L_{0} &=&%
    \begin{pmatrix}
    0 & 0.3333 \\
    0.3333 & 0%
    \end{pmatrix}
    \\
    L_{1} &=&%
    \begin{pmatrix}
    0.5 & 1.4107 \\
    1.4107 & 0.5%
    \end{pmatrix}
    \\
    L_{2} &=&%
    \begin{pmatrix}
    0.5 & 0.2560 \\
    0.2560 & 0.5%
    \end{pmatrix}%
    \end{eqnarray*}%
    are  obtained.

Considering the multiset formed by the entries of $E_{3}, 3$ is the least Perron root that can be considered because the trace becomes negative if the spectral radius is diminished.
\end{example}

\textbf{Acknowledgments}.
Enide Andrade was supported in part by the Portuguese Foundation for Science and Technology (FCT-Funda\c{c}\~{a}o para a Ci\^{e}ncia e a Tecnologia), through CIDMA - Center for Research and Development in Mathematics and Applications, within project UID/MAT/04106/2013. Hans Nina is supported by the projects UA INI-17-02 and FONDECYT 11170389.
M. Robbiano was partially supported by project VRIDT UCN 170403003. M. Robbiano also thanks the support of CIDMA - Center for Research and Development in Mathematics and Applications, within project UID/MAT/04106/2013 and the hospitality of the Mathematics Department of the University of Aveiro, Portugal, where this research was initiated.


\begin{thebibliography}{120}

\bibitem{Kenneth} K. Andrews, S. Dolinar, J. Thorpe, IEEE Xplore,
Conference: Information Theory, 2005. ISIT 2005. Proc. International Symposium on Information Theory and Its applications, 2005, DOI· 10.1109/ISIT.2005.1523758.

\bibitem{Baker} J. Baker, F. Hiergeist, G. E. Trapp, The structure of multi-blocks circulant, Kyungpook Math. J. 25 (1985): 71-75.
\bibitem{Berman} A. Berman, R. J. Plemmons, Nonnegative matrices in the Mathematical Sciences, SIAM Publications, Philadelphia, 1994.

\bibitem{Boro} A. Borobia. On nonnegative eigenvalue problem, Lin. Algebra
Appl. 223/224 (1995): 131-140, Special Issue honoring Miroslav Fiedler and
Vlastimil Pt\'{a}k.


\bibitem{Boyle} M. Boyle, D. Handelman, The spectra of nonnegative matrices
via symbolic dynamics, Ann. of Math. 133, 2 (1991): 249-316.


\bibitem{Chao} C. - Y. Chao, A remark on symmetric circulant matrices, Linear Algebra Appl. 103 (1988): 133-148.

\bibitem{Davis} P.J. Davis, Circulant matrices, John Wiley \& Sons, New York, Chichester, Brisbane, Toronto (1979).

\bibitem{Friedland} S. Friedland, On an inverse problem for nonnegative and
eventually nonnegative matrices, Israel T. Math. 1, 29 (1978): 43-60.

\bibitem{Fiedler} M. Fiedler, Eigenvalues of nonnegative symmetric
matrices, Lin. Algebra Appl. 9 (1974): 119-142.

\bibitem{Georgious} S. D. Georgious, E. Lappas, Self-dual codes from circulant matrices. Des.
Codes Cryptogr. 64 (2012): 129-141.

\bibitem{GUO} W. Guo, Eigenvalues of nonnegative matrix, Linear Algebra and its Applications 266 (1997): 261--270.


\bibitem{Johnson3} H. Horn, C. R. Johnson, Topics in Matrix Analysis, Cambridge University Press, Cambridge, 1991.

\bibitem{Johnson2} C. R. Johnson, C. Mariju\'an, P. Paparella, M. Pisonero, The NIEP, https://arXiv:1703.10992v2 [math.SP], 2017.


\bibitem{Johnson} C. R. Johnson, Row stochastic matrices similar to doubly
stochastic matrices, Lin. and Multilin. Algebra 2 (1981): 113-130.

\bibitem{Laffey} C. Johnson, T. Laffey, R. Loewy, The real and symmetric
nonnegative inverse eigenvalue problems are different, Proc. Amer. Math
Soc.,12, 124 (1996): 3647-3651.

\bibitem{Kang} R. D. Kangwai, S. D. Guest, S. Pellegrino, An introduction to the analysis of symmetric structures, Computers and Structures 71 (1999): 671--688.

\bibitem{Karlin} M. Karlin, New binary coding results by circulant, IEEE Trans. Inform. Theory, 15, (1969): 81-92.

\bibitem{Karner} H. Karner, J. Schneid, C. W. Ueberhuber, Spectral decomposition of real circulant matrices. Lin. Algebra Appl. 367 (2003): 301-311.

\bibitem{Laffey1}T. Laffey, Extreme nonnegative matrices,\ Lin. Algebra
Appl. 275/276 (1998): 349-357. Proceedings of the sixth conference of the
international Linear Algebra Society (Chemnitz, 1996).

\bibitem{Laffey2} T. Laffey, Realizing matrices in the nonnegative inverse
eigenvalue problem, Matrices and group representations (Coimbra, 1998),
Textos Mat. S\'{e}r. B, 19, Univ. Coimbra, Coimbra, (1999): 21-31.

\bibitem{Laffey-Smigoc} T. Laffey, H. \v{S}migoc, Nonnegative realization of spectra having negative real parts, Linear Algebra Appl., 384 (2004): 199--206.

\bibitem{LwyLdn} R. Loewy, D. London, A note on an inverse problem for
nonnegative matrices, Lin. and Multilin. Algebra 6, 1 (1978/79): 83-90.

\bibitem{LMc} R. Loewy,  J. J. Mc Donald, The symmetric nonnegative inverse
eigenvalue problem for $5\times 5$ matrices, Linear Algebra Appl. 393
(2004): 275-298.

\bibitem{MAR} C. Manzaneda, E. Andrade, M. Robbiano, Realizable lists via the spectra of structured matrices, Lin. Algebra Appl. 534 (2017): 51-72.

\bibitem{Mayo} J. Mayo Torre, M. R. Abril, E. Alarcia Est\'{e}vez, C. Mariju\'{a}n, M. Pisonero, The nonnegative inverse problem from the coeficientes of the characteristic polynomial EBL digraphs, Linear Algebra Appl. 426 (2007): 729-773.

\bibitem{Meehan}  M. E. Meehan, Some results on matrix spectra, Phd thesis, National University of Ireland, Dublin, 1998.

\bibitem{PP} P. Paparella, Realizing Sule\u{\i}manova-type spectra
via permutative matrices, Electron. J. Linear Algebra, 31 (2016): 306-312.

\bibitem{HzlP} H. Perfect, On positive stochastic matrices with real
characteristic roots, Proc. Cambridge Philos. Soc. 48 (1952): 271-276.

\bibitem{RSGuo} O. Rojo, R. L. Soto, Guo perturbations for symmetric nonnegative circulant
matrices, Linear Algebra Appl. 431 (2009): 594-607.

\bibitem{RS1} O. Rojo, R. L. Soto, Applications of a Brauer Theorem in the nonnegative inverse eigenvalue problem, Linear Algebra Appl. 416 (2007): 1-18.

\bibitem{RS} O. Rojo, R. L. Soto, Existence and construction of nonnegative
matrices with complex spectrum, Linear Algebra Appl. 368 (2003): 53–69.

\bibitem{Smith} R. L. Smith, Moore- Penrose Inverses of block circulant and block $k$-circulant matrices, Linear Algebra Appl. 16 (1977): 237-245.

\bibitem{SRM} R. Soto, O. Rojo, C. Manzaneda, On the nonnegative realization of partitioned spectra, Electron. J.  Linear Algebra, 22 (2011): 557-572.

\bibitem{Soules} G. W.  Soules, Constructing symmetric nonnegative matrices,
Linear and Multilin. Algebra 13, 3 (1983): 241-251.

\bibitem{Smgoc1} H. \v{S}migoc, The inverse eigenvalue problem for
nonnegative matrices, Linear Algebra Appl. 393 (2004): 365-374.

\bibitem{Smigoc2} H. \v{S}migoc, Construction of nonnegative matrices and the inverse
eigenvalue problem, Lin. and Multilin. Algebra 53, 2 (2005): 85-96.

\bibitem{SLMNva}H. R.  Sule\u{\i}manova, Stochastic matrices with real
characteristic numbers, Doklady, Akad. Nuk SSSR (N. S.) 66 (1949): 343-345.

\bibitem{Trapp} G. E. Trapp, Inverses of circulant matrices and block circulant matrices, Kyungpook Math. J. 13 (1973) 11-20.

\bibitem{Zhan} F. Zhang, Matrix Theory, Basic Results and Techniques. Second Edition. Universitext. Springer. http://www.springer.com/series/223.

\end{thebibliography}
\end{document}